\documentclass{amsproc}
\usepackage{graphicx}
\usepackage{amscd}

\theoremstyle{plain}

\newtheorem{remark}{Remark}

\newtheorem{theorem}{Theorem}
\numberwithin{equation}{section}

\begin{document}
\title[Discontinuous multiplication]{Multiplication is discontinuous in the Hawaiian earring group (with the
quotient topology)}
\author{Paul Fabel}
\address{Drawer MA, Mississippi State, MS 39762}
\email[Fabel]{fabel@ra.edu}
\urladdr{http://www2.msstate.edu/\symbol{126}fabel/}
\date{December 1 2009}

\begin{abstract}
The natural quotient map $q$ from the space of based loops in the Hawaiian
earring onto the fundamental group provides a new example of a quotient map
such that $q\times q$ fails to be a quotient map. This provides a
counterexample to the question of whether the fundamental group (with the
quotient topology) of a compact metric space is always a topological group
with the standard operations.
\end{abstract}

\maketitle

\section{Introduction}

Following the definitions in \cite{Biss}, there is a natural quotient
topology one can impart on the familiar based fundamental group $\pi
_{1}(X,p)$ of a topological space $X.$

If $L(X,p)$ denotes the space of $p$ based loops in $X$ with the compact
open topology, and if $q:L(X,p)\rightarrow \pi _{1}(X,p)$ is the natural
surjection, endow $\pi _{1}(X,p)$ with the quotient topology such that $%
A\subset \pi _{1}(X,p)$ is closed in $\pi _{1}(X,p)$ if and only of $%
q^{-1}(A)$ is closed in $L(X,p).$

As pointed out in \cite{Fab1} and \cite{Calcut}, it became an open question
whether or not group multiplication is always continuous in $\pi _{1}(X,p),$
various papers have stumbled on the issue, and in the paper at hand we
settle the question in the negative via counterexample (Theorem \ref{main1})
in which $X$ is the compact metric space the Hawaiian earring.

It is well known that outside the category of k-spaces ( compactly generated
spaces \cite{Steen}) if $q:Y\rightarrow Z$ is a quotient map, the map $%
q\times q:Y\times Y\rightarrow Z\times Z$ can fail to be a quotient map.
However familiar counterexamples of such phenomenon (Ex. 8 p. 141 \cite
{Munkres}) necessarily involve somewhat strange spaces, since for example it
cannot happen that both $Y$ and $Z$ are metrizable.

In this paper we take $X=HE$ to be the Hawaiian earring, the union of a null
sequence of circles joined at the common point $p$.

In the bargain we obtain a naturally occurring example of a quotient\ map $%
q:Y\rightarrow Z$ such that $q\times q:Y\times Y\rightarrow Z\times Z$ fails
to be a quotient map (Theorem \ref{main1}).

The example of the Hawaiian earring (a non locally contractible 1
dimensional Peano continuum) is in a sense the simplest counterexample,
since in general if $X$ is a locally contractible Peano continuum then $\pi
_{1}(X,p)$ is a discrete topological group ( \cite{Fab1} \cite{Calcut}).

Independently Brazas \cite{Brazos} has recently discovered some other
interesting counterexamples of a different flavor in which $X$ fails to be a
regular $(T_{3})$ space \cite{Brazos}.

Moreover the issue of regularity $(T_{3})$ also rears now its head as
follows. As a consequence of Theorem \ref{main1} we know $\pi _{1}(HE,p)$ is
a totally disconnected Hausdorff space but \textbf{not} a topological group
with the familiar operations. This leads to the open questions ``Is $\pi
_{1}(HE,p)$ regular?'', ``What is the inductive dimension of $\pi
_{1}(HE,p)?"$

\section{Main result and implications}

The \textbf{Hawaiian earring }$HE$ is the union of a null sequence of
circles joined at a common point $p.$

Formally $HE$ is the following subspace of the plane $R^{2},$ for an integer 
$n\geq 1$ let $X_{n}$ denote the circle of radius $\frac{1}{n}$ centered at $%
(\frac{1}{n},0)$ and define $HE=\cup _{n=1}^{\infty }X_{n}$.

Let $p=(0,0)$, let $Y_{0}=\{p\},$ and let $Y_{N}=\cup _{n=0}^{N}X_{n}.$

Let $\lim_{\leftarrow }Y_{n}$ denote the inverse limit space determined by
the natural retractions $r_{n}:Y_{n+1}\rightarrow Y_{n},$ such that $%
r_{n|Y_{n}}=id_{|Y_{n}}$ and $r_{n}(X_{n+1})=\{p\}.$

Thus, (as a subspace with the product topology), $\lim_{\leftarrow }Y_{n}$
denotes the space of all sequences $(y_{0},y_{1},y_{2},...)\subset
Y_{0}\times Y_{1}\times Y_{2}...$ such that $y_{n}=r_{n}(y_{n+1})$).

Consider the canonical homeomorphism $j:HE\rightarrow \lim_{\leftarrow
}Y_{n} $ $\ $such that $j(x)=(p,p,...p,x,x,x...).$

If $\phi :\pi _{1}(HE,p)\rightarrow \lim_{\leftarrow }\pi _{1}(Y_{n},p)$ is
the canonical homomorphism induced by $j$, it is a well known but nontrivial
fact that $\phi $ is one to one \cite{MM}.

Moreover as shown in (\cite{Conner} \cite{desmit}), elements of $\pi
_{1}(HE,p)$, can be understood precisely as the `irreducible' infinite words
in the symbols $\{x_{1},x_{2},..\}$ such that no symbol appears more than
finitely many times.

In particular $\pi _{1}(HE,p)$ naturally contains a subgroup canonically
isomorphic to the free group over symbols $\{x_{1},x_{2},..\}$.

Let $L(HE,p)$ denote the space of maps $f:[0,1]\rightarrow HE$ such that $%
f(0)=f(1)=p$ and impart $L(HE,p)$ with the topology of uniform convergence.

Let $q:L(HE,p)\rightarrow \pi _{1}(HE,p)$ denote the canonical quotient map
such that $q(f)=q(g)$ if and only if $f$ and $g$ are path homotopic in $HE.$

Let $q_{n}=(\frac{2}{n},0)\in X_{n}.$ Define the \textbf{oscillation number }%
$O_{n}:L(HE,p)\rightarrow \{0,1,2,3,4,...\}$ to be the maximum number $m$
such that there exists a set $T=\{0,t_{1},t_{2},...,t_{2m}\}\subset \lbrack
0,1]$ such that $0<t_{1}<t_{2}..<t_{2m}=1$ with $f(t_{2i})=p$ and $%
f(t_{2i+1})=q_{n}.$

Uniform continuity of $f$ over $[0,1]$ ensures that $O_{n}(f)<\infty $.
Other elementary properties of $O_{n}$ are observed or established in the
following remarks. (See also \cite{fab4}).

\begin{remark}
\label{osc} Suppose $f_{k}\rightarrow f$ uniformly in $L(HE,p)$ and $%
O_{n}(f_{k})\geq m$ as shown by the sets $T_{k}\subset \lbrack 0,1]$ such
that $\left| T_{k}\right| =2m+1.$ Then if $T\subset \lbrack 0,1]$ is a
subsequential limit of $\{T_{k}\}$ in the Hausdorff metric, then $T$ shows $%
O_{n}(f)\geq m.$
\end{remark}

\begin{remark}
\label{dom} Suppose $f$ and $g$ are in the same path component of $L(HE,p)$
and suppose $g:[0,1]\rightarrow Y_{m}$ is a geodesic corresponding to a
maximally reduced finite word $w$ in the free group $F_{m}$ on $m$ letters.
Then $O_{n}(f)\geq O_{n}(g).$ ( Replacing $f$ by $f_{1}=r_{m}(f)$ we have $%
O_{n}(f)\geq O_{n}(f_{1}).$ For each nontrivial interval $J\subset
im(f_{1})\backslash \{p\}$ such that $f_{1\overline{J}}$ is an inessential
based loop, replace $f_{J}$ by the constant map to obtain $f_{2}$ and
observe $O_{n}(f_{2})\leq O_{n}(f_{1})$ and let $v\in F_{m}$ denote the
(unreduced) finite word corresponding to $f_{2}$. Then $\left| v\right| \geq
\left| w\right| $ and hence $O_{n}(f_{2})\geq O_{m}(g).)$
\end{remark}

\begin{remark}
\label{onetoone}Since $\phi :\pi _{1}(HE,p)\rightarrow \lim_{\leftarrow }\pi
_{1}(Y_{n},p)$ is one to one, the path components of $L(HE,p)$ are closed
subspaces of $L(HE,p)$. (If $\{f_{n}\}$ is a sequence of homotopic based
loops in $HE,$ and if $f_{n}\rightarrow f$ and then $r_{N}(f_{n})\rightarrow
r_{N}(f)$ for each $N.$ Since $Y_{N}$ is locally contractible $r_{N}(f_{n})$
is eventually path homotopic to $r_{N}(f).$ Thus, since $\phi $ is one to
one, $f$ is path homotopic to $f_{1}$ in $HE).$ See also the proof of
Theorem 2 \cite{fab4}.
\end{remark}

\begin{remark}
\label{comp} If $Z$ is a metric space such that each path component of $Z$
is a closed subspace of $Z,$ then each path component of $Z\times Z$ is a
closed subspaces of $Z\times Z.$ ( If $(x_{n},y_{n})\rightarrow (x,y)$ and $%
\{(x_{n},y_{n})\}$ is in a path component of $Z\times Z$ then obtain paths $%
\alpha $ and $\beta $ in $Z$ connecting $x$ to $\{x_{n}\}$ and $y$ to $%
\{y_{n}\}$ and $(\alpha ,\beta )$ is the desired path in $Z\times Z$).
\end{remark}

What follows is the main result of the paper.

\begin{theorem}
\label{main1} The product of quotient maps $q\times q:L(HE,p)\times
L(HE,p)\rightarrow \pi _{1}(HE,p)\times \pi _{1}(HE,p)$ fails to be a
quotient map, standard multiplication ( by path class concatenation) $M:\pi
_{1}(HE,p)\times \pi _{1}(HE,p)\rightarrow \pi _{1}(HE,p)$ is discontinuous,
and the fundamental group $\pi _{1}(HE,p)$ fails to be a topological group
with the standard group operations.
\end{theorem}

\begin{proof}
Let $x_{n}\in L(HE,p)$ orbit $X_{n}$ once counterclockwise.

Applying path concatenation, for integers $n\geq 2$ and $k\geq 2$ let $%
a(n,k)\in L(HE,p)$ be a based loop corresponding to the finite word $%
(x_{n}x_{k}x_{n}^{-1}x_{k}^{-1})^{k+n}$ and let $w(n,k)\in L(HE,p)$ be a
based loop corresponding to the finite word $%
(x_{1}x_{k}x_{1}^{-1}x_{k}^{-1})^{n}.$

Let $F\subset \pi _{1}(HE,p)\times \pi _{1}(HE,p)$ denote the set of all
doubly indexed ordered pairs $([a(n,k)],[w(n,k)]).$

Let $P\in L(HE,p)$ denote the constant map such that $f([0,1])=\{p\}.$

To prove $q\times q$ fails to be a quotient map it suffices to prove that $F$
is not closed in $\pi _{1}(HE,p)\times \pi _{1}(HE,p)$ but that $(q\times
q)^{-1}(F)$ is closed in $L(HE,p)\times L(HE,p).$

To prove $F$ is not closed in $\pi _{1}(HE,p)\times \pi _{1}(HE,p)$ we will
prove that $([P],[P])\notin F$ but $([P],[P])$ is a limit point of $F.$

Recall $\phi :\pi _{1}(HE,p)\rightarrow \lim_{\leftarrow }\pi _{1}(Y_{n},p)$
is one to one and $k\geq 2.$ Thus $[P]\neq \lbrack w(n,k)]$ and $[P]\neq
\lbrack a(n,k)].$ Thus $([P],[P])\notin F.$

Suppose $[P]\in U$ and $U$ is open in $\pi _{1}(HE,p).$ Let $V=q^{-1}(U).$
Then $V$ is open in $L(HE,p)$ since, by definition, $q$ is continuous.

Note $P\in V.$ Thus there exists $N$ and $K$ such that if $n\geq N$ and $%
k\geq K$ then $a(n,k)\in V.$

Note $(x_{1}x_{1}^{-1})^{N}$ is path homotopic to $P$ and hence $%
(x_{1}x_{1}^{-1})^{N}\in V.$ Note (suitably parameterized over $[0,1]$), $%
w(N,k)\rightarrow (x_{1}x_{1}^{-1})^{N}$ uniformly in $L(HE,p).$

Thus there exists $K_{2}\geq K$ such that if $k\geq K_{2}$ then $w(N,k)\in
V. $ Hence $([w(N,K_{2})],[a(N,K_{2})])\in U\times U.$

This proves $([P],[P])$ is a limit point of $F,$ and thus $F$ is not closed
in $\pi _{1}(HE,p)\times \pi _{1}(HE,p).$

To prove $(q\times q)^{-1}(F)$ is closed in $L(HE,p)\times L(HE,p)$ suppose $%
(f_{m},g_{m})\rightarrow (f,g)$ uniformly and $(f_{m},g_{m})\in (q\times
q)^{-1}(F).$

Note $O_{1}(w(n,k))=2n$ and $O_{N}(a(N,k))\geq 2(N+k).$

Let $a(n_{m},k_{m})$ and $w(n_{m},k_{m})$ be path homotopic to respectively $%
f_{m}$ and $g_{m}.$

By Remark \ref{dom} $O_{1}(g_{m})\geq O_{1}(w(n_{m},k_{m}))=2n_{m}.$

Thus if $\{n_{m}\}$ contains an unbounded subsequence then, by Remark \ref
{osc}, $O_{1}(g)\geq \lim \sup O_{1}(w(n_{m},k_{n_{m}}))=\infty $ and we
have a contradiction since $O_{1}(g)<\infty .$

Thus $\{n_{m}\}$ is bounded and the sequence $\{n_{m}\}$ takes on finitely
many values.

In similar fashion, if $\{k_{m}\}$ is unbounded then there exists $N$ and a
subsequence $\{k_{m_{l}}\}$ such that $O_{N}(a(N,k_{m_{l}})\rightarrow
\infty .$ Thus $O_{N}(f)\geq \lim \sup O_{N}(a(N,k_{m_{l}}))=\infty $
contradicting the fact that $O_{N}(f)<\infty .$

Thus both $\{n_{m}\}$ and $\{k_{m}\}$ are bounded and hence (by the pigeon
hole principle) there exists a path component $A\subset L(HE,p)\times
L(HE,p) $ containing a subsequence $(f_{m_{l}},g_{m_{l}}).$

It follows from Remarks \ref{onetoone} and \ref{comp} that $(f,g)\in A.$

Thus $(q\times q)^{-1}(F)$ is closed and hence $q\times q$ fails to be a
quotient map.

In similar fashion we will prove that group multiplication $M:\pi
_{1}(HE,p)\times \pi _{1}(HE,p)\rightarrow \pi _{1}(HE,p)$ is discontinuous,
and hence $\pi _{1}(HE,p)$ will fail to be a topological group with the
standard group operations. To achieve this we will exhibit a closed set $%
A\subset \pi _{1}(HE,p)$ such that $M^{-1}(A)$ is not closed in $\pi
_{1}(HE,p)\times \pi _{1}(HE,p).$

Consider the doubly indexed set $A=M(F)\subset \pi _{1}(HE,p)$ such that
each element of $A$ is of the form $[a(n,k)]\ast \lbrack w(n,k)]$ (with $%
\ast $ the denoting familiar path class concatenation).

On the one hand observe by definition (and since $\phi $ is one to one) $%
[a(n,k)]\ast \lbrack w(n,k)]\neq \lbrack P].$

Thus $[P]\notin A$ and $([P],[P])\notin M^{-1}(A).$ Note $F\subset M^{-1}(A)$
and by the previous argument $([P],[P])$ is a limit point of $F.$ Thus $%
M^{-1}(A)$ is not closed in $\pi _{1}(HE,p)\times \pi _{1}(HE,p).$

On the other hand we will prove $A$ is closed in $\pi _{1}(HE,p)$ by proving 
$q^{-1}(A)$ is closed in $L(HE,p).$ Suppose $f_{m}\rightarrow f\in L(HE,p)$
and $f_{m}\in q^{-1}(A).$

Obtain $n_{m}$ and $k_{m}$ such that $f_{m}\in \lbrack a(n,k)]\ast \lbrack
w(n,k)].$

In similar fashion to the previous proof, if $\{n_{m}\}$ is unbounded we
obtain the contradiction $O_{1}(f)\geq \lim \sup O_{1}(f_{m})=\infty .$

If $\{n_{m}\}$ is bounded and $\{k_{m}\}$ is unbounded we obtain $N$ and a
subsequence $k_{m_{l}}$ and the contradiction $O_{N}(f)\geq \lim \sup
O_{N}(f_{m_{l}})=\infty .$

Thus both $\{n_{m}\}$ and $\{k_{m}\}$ are bounded. It follows by the pigeon
hole principle that some path component $B\subset L(HE,p)$ contains a
subsequence $\{f_{m_{l}}\}$ and it follows from Remark \ref{onetoone} that $%
f\in B.$ Hence $q^{-1}(A)$ is closed in $L(HE,p)$ and thus $A$ is closed in $%
\pi _{1}(HE,p).$
\end{proof}

\end{document}